\documentclass[12pt]{amsart}
 \usepackage[dvips]{epsfig}
 \usepackage{comment}
 \usepackage{enumerate}
 \usepackage{amsgen, amstext,amsbsy,amsopn, amsthm, amsfonts,amssymb,amscd,amsmat
 h,euscript,enumerate,url,verbatim,calc,xypic, mathtools,}
\usepackage{textcomp}
 \usepackage{latexsym}
 \usepackage{graphics}
 \usepackage{color}

\newcommand{\Pic}{\rm Pic\,}
\newcommand{\proset}{\,\mathrel{\lower 4pt\hbox{$\scriptscriptstyle/$}
\mkern -14mu\subseteq }\,} 
 
 \newtheorem{theorem}{Theorem}[section]
  \newtheorem{corollary}[theorem]{Corollary}
 \newtheorem{lemma}[theorem]{Lemma}
 \newtheorem{proposition}[theorem]{Proposition}

\newtheorem{remark}[theorem]{Remark}

 \newtheorem{example}[theorem]{Example}

\numberwithin{equation}{section}

\def\ker{\operatorname{ker}}
\def\coker{\operatorname{coker}}

\usepackage{amsmath}
 \makeatother 
 
\begin{document}
\date{\today}

\title{Equivariant Picard groups and Laurent polynomials}
 \author{Vivek Sadhu} 
 
 \address{Department of Mathematics, Indian Institute of Science Education and Research Bhopal, Bhopal Bypass Road, Bhauri, Bhopal-462066, Madhya Pradesh, India}
 \email{ vsadhu@iiserb.ac.in, viveksadhu@gmail.com}
 \thanks{Author was supported by SERB-DST MATRICS grant MTR/2018/000283}
 
 \begin{abstract}
 Let $G$ be a finite group. For a $G$-ring $A,$ let ${\rm Pic}^{\it G}({\it A})$ denote the equivariant Picard group of $A.$ We show that if $A$ is a finite type algebra over a field $k$ then ${\rm Pic}^{\it G}({\it A})$ is contracted in the sense of Bass with contraction $H_{et}^{1}(G; Spec(A), \mathbb{Z}).$ This gives a natural decomposition of the group ${\rm Pic}^{\it G}({\it A[t, t^{-1}]}).$
 \end{abstract}
 \maketitle
\section{Introduction}
Throughout the paper, $G$ is a finite group and a $G$-scheme will always mean a scheme which is separated, finite type over a field $k$ equipped with an action of $G.$

H. Bass introduced the notion of contracted functor to study $K$-theory of Laurent polynomial rings (see chapter XII of \cite{Bass}). In fact, Bass negative $K$-groups are defined using this notion. In \cite{Bass}, Bass also observed that the study of the notion of contracted functor for Picard groups may be useful to understand negative $K$-groups. In \cite{Wei}, Weibel studied the same for Picard groups and proved that the functor $\Pic$ is contracted on the category of schemes in the sense of Bass, i.e., for every scheme $X,$ there is a natural decomposition $$\Pic({\it X[t, t^{-1}}])\cong \Pic({\it X}) \oplus {\it N}\Pic({\it X}) \oplus {\it N}\Pic({\it X}) \oplus {\it H}_{et}^{1}({\it X}, \mathbb{Z}),$$ where $X[t, t^{-1}]= X \times \mathbb{G}_{m},$ $N\Pic({\it X})=\ker[ \Pic({\it X }\times \mathbb{A}^{1})\to \Pic({\it X})]$ and $\mathbb{Z}$ denotes the constant \'etale sheaf on $X.$ The goal of this article is to prove an analogous result in the equivariant setting. The equivariant Picard group of a $G$-scheme $X,$ denoted by ${\rm Pic}^{\it G}({\it X}),$ is the group of isomorphism classes of $G$-linearized line bundle on $X.$ We show that the functor ${\rm Pic}^{\it G}$ is contracted on the category of $G$-schemes in the sense of Bass. 
More precisely, we prove the following (Theorem \ref{Picg contracted} and Theorem \ref{NL^G=0}):
\begin{theorem}\label{thm1}
 Let $X$ be a $G$-scheme. We write $X[t]$ and $X[t, t^{-1}]$ for $X \times_{k} \mathbb{A}^{1}_{k}$ and $X \times_{k} \mathbb{G}_{m}.$ Then there is a natural split exact sequence
 \begin{equation*}
  0\to {\rm Pic}^{\it G}({\it X}) \to {\rm Pic}^{\it G}({\it X[t]})\oplus {\rm Pic}^{\it G}({\it X[t^{-1}]}) \to {\rm Pic}^{\it G}({\it X[t, t^{-1}]})\to {\it L}{\rm Pic}^{\it G}({\it X}) \to 0
 \end{equation*} with $L{\rm Pic}^{\it G}({\it X})\cong {\it H}_{\mathrm{et}}^{1}(G; X; \mathbb{Z}).$ Moreover, $$L{\rm Pic}^{\it G}({\it X})\cong {\it L}{\rm Pic}^{\it G}({\it X[t]})\cong {\it L}{\rm Pic}^{\it G}({\it X[t, t^{-1}]}).$$
\end{theorem}

The organization of this article is as follows.

In section \ref{base section}, we recall the notion of contracted functor for the category of $G$-rings (resp. $G$-schemes). We also discuss some known examples of contracted functors. 

In section \ref{equi}, we first recall few basics pertaining to group scheme action and quotients by group actions. Next, we discuss equivariant sheaves and cohomology theories which will play an important role to prove our results in the rest of the paper. 

In section \ref{equi swan}, we mainly study the homotopy invariance of the group
 ${\rm Pic}^{\it G}({\it X}).$ A classical theorem of Traverso's says that if $A$ is a seminormal ring then the natural map $\Pic({\it A}) \to \Pic({\it A[t_{1}, t_{2}, \dots, t_{r}]})$ is an isomorphism for $r\geq 0.$ We prove an equivariant version of Traverso's theorem. The following is our main result of section \ref{equi swan} (see Theorem \ref{hom inv for equipic}):
 \begin{theorem}
 Let $X$ be a seminormal $G$-scheme. Then the natural map $$ {\rm Pic}^{\it G}({\it X})\to {\rm Pic}^{\it G}({\it X[t_1, t_2, \dots, t_r]})$$ is an isomorphism for all $r\geq 0.$
\end{theorem}

In section \ref{main sec}, we study the contractibility of the functor ${\rm Pic}^{\it G}.$ More explicitly, we prove Theorem \ref{thm1}. For every $G$-scheme $X,$ let $N{\rm Pic}^{\it G}(X)= \ker [{\rm Pic}^{\it G}(X\times_{k} \mathbb{A}_{k}^{1}) \to {\rm Pic}^{\it G}(X)].$ Then there is a natural decomposition (by Theorem \ref{thm1}) $${\rm Pic}^{\it G}({\it X[t, t^{-1}}])\cong {\rm Pic}^{\it G}({\it X}) \oplus {\it N}{\rm Pic}^{\it G}({\it X}) \oplus {\it N}{\rm Pic}^{\it G}({\it X}) \oplus {\it H}_\mathrm{et}^{1}({\it G}; {\it X}, \mathbb{Z}).$$  We also deduce a general formula for the group ${\rm Pic}^{\it G}({\it X[t_{1}, t_{1}^{-1}, t_{2}, t_{2}^{-1}, \dots, t_{m}, t_{m}^{-1}]})$ (see Corollary \ref{formula for Pic^g}). We observe that the cohomological interpretation of the term ${\it L}{\rm Pic}^{\it G}({\it X})$ may fail for the Zariski topology (see Remark \ref{true for nis but not zar}). Further, we study the kernel of the forgetful map $\eta_{X}: {\rm Pic}^{\it G}({\it X}) \to \Pic({\it X}).$ We show that $\ker(\eta_{X})\cong H^{1}(G, H_{zar}^{0}(X, \mathcal{O}_{X}^{\times}))$ and $\ker(\eta_{X})$ is a contracted functor with contraction $H^{1}(G, H^{0}_\mathrm{et}(X, \mathbb{Z}))$ (see Theorem \ref{ker is contracted}). Finally, we discuss a vanishing criterion for the term $L{\rm Pic}^{\it G}(X).$ We show the following (see Corollary \ref{vanishing for hensel}):

\begin{theorem}
 Let $X={\rm Spec}(A)$ be a $G$-scheme. If $A$ is a hensel local ring then $L{\rm Pic}^{\it G}({\it X})=0.$
\end{theorem}

{\bf Acknowledgements:} The author is grateful to Charles Weibel for his valuable comments and suggestions during the preparation of this article. He would also like to thank Charanya Ravi for fruitful email exchanges. Finally, He would like to thank the referee for valuable comments and suggestions.

 \section{Contracted functors}\label{base section}
 The notion of contracted functor from the category of rings to abelian groups was introduced by H. Bass (see chapter XII of  \cite{Bass}). This notion also makes sense from many categories (e.g., commutative rings, schemes, ring extensions etc.) to any abelian category (e.g., abelian groups, modules, sheaves etc.).  Let us recall the notion of contracted functor from the category of $G$-rings to abelian groups. Here $G$ is any abstract group. 
 Recall that a ring $A$ is said to be $G$-ring if it has a left action of $G$ by ring automorphisms. Let $A$ and $B$ be two $G$-rings. A morphism from $A$ to $B$ is a ring homomorphism $f: A \to B$ such that $f(g.a)=g.f(a)$ for all $g \in G$ and $a\in A.$
 
 Let $F$ be a functor from the category of $G$-rings to abelian groups. We define functors $NF$ and $LF$ as follows:
 $$NF(A)=N_{t}F(A)= \ker [F(A[t]) \stackrel{t\mapsto 1}\to F(A)]\cong \coker [F(A) \stackrel{F(i_{+})}\to F(A[t])];$$
 $$LF(A)= \coker [F(A[t])\oplus F(A[t^{-1}]) \stackrel{{\rm add}}\to F(A[t, t^{-1}])].$$ Here the $G$-action on $A[t]$ induces from $A$ with $g.t=t$ for all $g\in G.$ Clearly, $F(A[t])\cong F(A)\oplus NF(A)$ because the inclusion $A\stackrel{i_{+}}\hookrightarrow A[t]$ has a $G$-linear section $A[t] \stackrel{t\mapsto 1}\longrightarrow A.$ By iterating these functors, one can define $N^{i}F$ and $L^{i}F$ for $i>0.$ More generally, we get the formula 
 \begin{equation}\label{formula for pol}
  F(A[t_1, t_2, \dots, t_r])\cong (1 + N)^{r}F(A) ~{\rm for}~ r\geq 0.
 \end{equation}
 We say that $F$ is a {\it acyclic} functor if the following sequence 
 \begin{equation}\label{main seq}
   0 \to F({\it A}) \stackrel{(+,-)}\to F({\it A[t]}) \times F({\it A[t^{-1}]}) \stackrel{add}\to F({\it A[t, t^{-1}]}) \to {\it L}F({\it A})\to 0
 \end{equation}
is exact for every $G$-ring $A.$ Here $(+, -)(x)=(F(i_{+})(x), -F(i_{-})(x))$ for $x\in F(A).$ We say that $F$ is a {\it contracted} functor if (\ref{main seq}) is naturally split exact, i.e., exact and there is a natural split map $LF(A) \to F(A[t, t^{-1}]).$ Note that if $F$ is a contracted functor then there is a natural decomposition $$ F(A[t, t^{-1}])\cong F(A)\oplus N_{t}F(A)\oplus N_{t^{-1}}F(A) \oplus LF(A).$$ By iterating, we get
\begin{equation}\label{gen form}
 F(A[t_1, t_{1}^{-1}, t_{2}, t_{2}^{-1}, \dots, t_{r}, t_{r}^{-1}])\cong (1 + 2N + L)^{r}F(A).
\end{equation}

Note that the notion of contracted functor can also be defined in a similar way from the category of $G$-schemes to abelian groups. Here are few examples of known contracted functors.
\begin{example}\label{ex1}{\rm
 For a scheme $X,$ let $U(X)= H_{zar}^{0}(X, \mathcal{O}_{X}^{\times}).$ The global units functor $U$ is contracted on the category of schemes with contraction $LU(X)\cong H_{zar}^{0}(X, \mathbb{Z})\cong H_\mathrm{et}^{0}(X, \mathbb{Z}).$ Moreover, $LU(X)\cong LU(X[t])\cong LU(X[t, t^{-1}]),$ i.e., $NLU(X)=L^2U(X)=0$ (see Proposition 7.2 of \cite{Wei}). The functor $\Pic$ is also contracted on the category schemes with contraction $L\Pic({\it X})\cong {\it H}_\mathrm{et}^{1}({\it X}, \mathbb{Z})\cong {\it H}_{\mathrm{nis}}^{1}({\it X}, \mathbb{Z})$ and $L\Pic({\it X})\cong {\it L}\Pic({\it X[t]})\cong {\it L}\Pic({\it X[t, t^{-1}]})$ (see Theorem 7.6 and Proposition 7.7 of \cite{Wei}).}
\end{example}

\begin{example}
 \label{ex3} {\rm For each $n,$ $K_{n}$ is a contracted functor on the category of quasi-projective scheme with contraction $LK_{n}=K_{n-1}$ (see Theorem V.8.3 of \cite{wei 1}).}
\end{example}

\begin{example}
 \label{ex4}{\rm Given a ring extension $f:A \hookrightarrow B,$ let $\mathcal{I}(f)$ be the multiplicative group of invertible $A$-submodule of $B.$ One can check that $\mathcal{I}$ is a functor from the category of ring extensions to abelian groups. The functor $\mathcal{I}$ is contracted with contraction $L\mathcal{I}(f)\cong H_\mathrm{et}^{0}(Spec(A), f_{*}\mathbb{Z}/\mathbb{Z})\cong H_{\mathrm{nis}}^{0}(Spec(A), f_{*}\mathbb{Z}/\mathbb{Z})$(see Theorem 5.1 of \cite{Sw}).
 Write $f[t]$(resp. $f[t, t^{-1}])$) for $A[t]\hookrightarrow B[t]$(resp. $A[t, t^{-1}]\hookrightarrow B[t, t^{-1}])$). Then, we have $L\mathcal{I}(f)\cong L\mathcal{I}(f[t])\cong L\mathcal{I}(f[t, t^{-1}])$(see Proposition 3.4 of \cite{Sw}). A map $f: X \to S$ of schemes is said to be faithful affine if it is affine and the structure map $\mathcal{O}_{S} \to f_{*}\mathcal{O}_{X}$ is injective. More generally, $\mathcal{I}$ can be thought as a functor from the category of faithful affine map of schemes to abelian groups. In fact, $\mathcal{I}$ is a contracted functor on the category of faithful affine map of schemes (see Theorem 5.2 of \cite{Sw}).}
\end{example}

 \section{Preliminaries on $G$-schemes and sheaves}\label{equi}
 Hereafter throughout this article, $k$ denotes a field and $G$ denotes a finite group unless otherwise stated. In this section, we briefly recall some basics on $G$-schemes and sheaves. The details can be found in \cite{FKM, SGA}. Let $Sch/k$ denote the category of separated, finite type schemes over $k.$ Let $Grp$ denote the category of groups. For a group $G,$ define a functor $$G_k: (Sch/k)^{op} \to Grp, T \mapsto G_{k}(T):= {\rm ~the~ group~ of~ locally~ constant~ maps~} f: T \to G,$$ 
 where $G$ has the discrete topology.  Note that $G_{k}$ is represented by a scheme $\amalg_{g\in G} {\rm Spec}(k).$ Hence,  $G$ can be viewed as a group scheme over $k$ whose underlying scheme is $\amalg_{g\in G} {\rm Spec}(k).$

 \subsection*{The category of $G$-schemes:} 
 Suppose $X \in Sch/k.$ Then a morphism $\sigma: G \times _{k} X \to X$ is called an action of $G$ on $X$ if for all $T \in Sch/k,$ the map $\sigma(T): G(T) \times X(T) \to X(T)$ on $T$-valued points defines an action of the group $G(T)$ on the set $X(T).$ We simply write $g.x$ for $\sigma (g, x).$ Let $Sch^{{\it G}}/k$ denote the category of $G$-schemes. More explicitly, an object $X\in Sch^{{\it G}}/k$ is an object of $Sch/k$ equipped with an action of $G$ on $X.$ A morphism between $G$-schemes $f: X \to Y$ is a $G$-linear morphism, i.e., it is a morphism in $Sch/k$ such that $f(g.x)= g.f(x).$
 
 \subsection*{Quotients by group actions:} Let $\sigma: G \times _{k} X \to X$ be a $G$-action on $X.$ Write $pr_{2}$ for the projection map $G\times_{k} X \to X.$ A morphism $q: X \to Y$ in $Sch/k$ is said to be $G$-invariant if $q\circ pr_{2}= q\circ \sigma,$ i.e., $q(g.x)= q(x)$ for all $x\in X(T)$ and $g\in G(T).$ By a quotients, we always mean a categorical quotients. We say that a morphism $q: X \to Y$ is a categorical quotient if $q$ is a $G$-invariant and $q$ is universal; which means that for any $G$-invariant morphism $q^{'}: X \to Y^{'}$ there is a unique $\alpha: Y \to Y^{'}$ such that $q^{'}= \alpha \circ q.$ Note that if a quotient exists then it must be unique upto unique isomorphism. We usually denote the quotient by $q: X \to X/G.$ If $X$ is quasi-projective over $k$ then the quotient $q: X \to X/G$ always exist (because in our case $G$ is a finite group). Moreover, we have the following properties:
 \begin{enumerate}
  \item $q$ is finite and surjective;
  \item $\mathcal{O}_{X/G}= q_{*}(\mathcal{O}_{X})^{{\it G}};$
  \item  If $X= {\rm Spec}(A)$ then $X/G= {\rm Spec} (A^{{\it G}}),$ where $A^{{\it G}}= \{a\in A|~ g.a=a {\rm~ for ~all~} g\in G\}.$
 \end{enumerate}

 \subsection*{$G$-sheaves and $G$-cohomology} Let $X$ be a $G$-scheme with an action map $\sigma.$ Let $\mathcal{F}$ be a $\tau$-sheaf of abelian group. Here $\tau$ is any one of the Zariski, \'etale and Nisnevich Grothendieck topologies on $X.$ We now recall few definitions for which $G$ is not necessarily a finite group. More generally, the following definition also makes sense for any algebraic group. A $G$-linearization of $\mathcal{F}$ is an isomorphism $\phi: \sigma^{*}\mathcal{F}\cong pr_{2}^{*}\mathcal{F}$ of sheaves on $G\times_{k} X$ with the following cocycle condition
  $$ pr_{23}^{*}(\phi)\circ (1 \times \sigma)^{*}(\phi)= (m \times 1)^{*} (\phi).$$ Here $m$ is the multiplication map $G \times_{k} G \to G$ and $pr_{23}$ is the projection to second and third factor $G \times_{k} G \times_{k} X \to G\times_{k} X.$ As in our case $G$ is a finite group, a $G$-linearization of $\mathcal{F}$ is equivalent to a family of isomorphisms $\phi_{g}: g^{*}\mathcal{F} \stackrel{\cong}\to \mathcal{F}$ for each $g\in G$ such that $\phi_{e}=id$ and $\phi_{gh}= \phi_{h}\circ h^{*}(\phi_{g})$ for all $g, h \in G.$ A $G$-sheaf in the $\tau$-topology is a pair $(\mathcal{F}, \phi),$ where $\mathcal{F}$ is a sheaf on $X$ and $\phi$ is a $G$-linearization of $\mathcal{F}.$ A $G$-module on $X$ is a $G$-sheaf $(\mathcal{F}, \phi),$ where $\mathcal{F}$ is an $\mathcal{O}_{X}$-module and the $G$-linearization $\phi: \sigma^{*}\mathcal{F}\cong pr_{2}^{*}\mathcal{F}$ is an isomorphism of $\mathcal{O}_{G\times X}$-modules.

  A morphism $f: (\mathcal{F}_{1}, \phi_{1}) \to (\mathcal{F}_{2}, \phi_{2})$ between $G$-sheaves is a morphism $f: \mathcal{F}_{1} \to \mathcal{F}_{2}$ of sheaves such that $pr_{2}^{*}f\circ \phi_{1}= \phi_{2}\circ \sigma^{*}f.$ We call such a morphism as equivariant morphism. The set of equivariant morphisms from $ (\mathcal{F}_{1}, \phi_{1})$ to $(\mathcal{F}_{2}, \phi_{2})$ is denoted by $Hom_{{\it G}}(\mathcal{F}_{1}, \mathcal{F}_{2}).$ 
  
Let $Ab_{\tau}(G, X)$ denote the category of $G$-sheaves on $X$ in the topology $\tau$ whose objects are $G$-sheaves on $X$ in the topology $\tau$ and morphisms are the equivariant morphisms. The category $Ab_{\tau}(G, X)$ is abelian and it has enough injectives. If $\mathcal{F}$ is a $G$-sheaf then then the group $G$ acts naturally on the space of global sections $\Gamma (X, \mathcal{F}).$ Let $Ab$ denote the category of abelian groups. The invariant global section functor is defined by 
$$ \Gamma_{X}^{{\it G}}: Ab_{\tau}(G, X) \to Ab,  \mathcal{F} \mapsto \Gamma(X, \mathcal{F})^{{\it G}},$$ where $\Gamma(X, \mathcal{F})^{{\it G}}= \{x\in \Gamma(X, \mathcal{F})| g.x=x ~{\rm ~for~ all}~ g\in G\}.$ Note that $\Gamma_{X}^{{\it G}}= (-)^{{\it G}}\circ \Gamma(X, -)$ is a left-exact functor. Then the $\tau$-$ G$-cohomology groups $H_{\tau}^{p}(G; X; \mathcal{F})$ are defined as the right derived functors $H_{\tau}^{p}(G; X; \mathcal{F}):= R^{p}\Gamma_{X}^{{\it G}}\mathcal{F}.$ Consider the following commutative diagram
\label{commutativity}
     $$\xymatrixcolsep{5pc}\xymatrix{
 Ab_{\tau}(G, X) \ar[rd]^{\Gamma_{X}^{{\it G}}} \ar[r]^{\Gamma(X, -)} & Ab \ar[d]^{(-)^{{\it G}}} \\
 &  Ab} $$of categories. The functor $\Gamma(X, -)$ sends injective $G$-sheaves to injective $G$-modules. Now the Grothendieck spectral sequence for the composition $(-)^{{\it G}}\circ \Gamma(X, -)$ gives the first quadrant convergent spectral sequence 
 $$E_{2}^{pq}= H^{p}(G, H_{\tau}^{q}(X, \mathcal{F}))\Rightarrow H_{\tau}^{p+q}(G; X; \mathcal{F}),$$ where $H^{*}(G,-)$ denotes the group cohomology. So, we have the five-term exact sequence
 \small\begin{equation}\label{5-term 1}
  0 \to H^{1}(G, H_{\tau}^{0}(X, \mathcal{F})) \to H_{\tau}^{1}(G; X; \mathcal{F}) \to H^{0}(G, H_{\tau}^{1}(X, \mathcal{F})) \to H^{2}(G, H_{\tau}^{0}(X, \mathcal{F})) \to H_{\tau}^{2}(G; X; \mathcal{F}).
 \end{equation}\normalsize

 \section{Equivariant Picard groups and Seminormality}\label{equi swan}
 
 In this section, we show that for a seminormal $G$-scheme, the homotopy invariance of the equivariant Picard group holds. Let us begin by recalling the definition of equivariant Picard  groups.
 
 It is wellknown that the tensor product of two $G$-linearized line bundles over a $G$-scheme $X$ is also a $G$-linearized line bundle. The dual of any $G$-linearized line bundle is $G$-linearized as well (see p.32 of \cite{FKM}). Thus, the isomorphism classes of $G$-linearized line bundles over $X$ form an abelian group with respect to the tensor product. We call it {\it equivariant Picard group} of $X$ and denote it by ${\rm Pic}^{\it G}({\it X}).$ A $G$-linear map $f: X \to Y$  always induces a group homomorphism $f^*: {\rm Pic}^{\it G}({\it Y}) \to {\rm Pic}^{\it G}({\it X})$ by sending $[(\mathcal{L}, \phi)]$ to $[(f^{*}\mathcal{L},(id_G \times f)^{*} \phi)].$ In fact, ${\rm Pic}^{\it G}$ defines a functor from the category of $G$-schemes to abelian groups. 

 \begin{lemma}\label{coh interpre}
  Let $X$ be a $G$-scheme. There are natural isomorphism
  $$ {\rm Pic}^{\it G}({\it X})\cong {\it H_{zar}^{1}(G; X; \mathcal{O}_{X}^{\times})}\cong {\it H_\mathrm{et}^{1}(G; X; \mathcal{O}_{X}^{\times})} \cong {\it H_{\mathrm{nis}}^{1}(G; X; \mathcal{O}_{X}^{\times})}.$$
 \end{lemma}
 \begin{proof}
  See Theorem 2.7 of \cite{HVO}.
 \end{proof}

 Thus, we have an exact sequence (by (\ref{5-term 1}))\small
 \begin{equation}\label{5-term 2}
  0 \to H^{1}(G, H_{\tau}^{0}(X, \mathcal{O}_{X}^{\times})) \to{\rm Pic}^{\it G}({\it X}) \to (\Pic({\it X}))^{{\it G}} \to {\it H^{2}(G, H_{\tau}^{0}(X, \mathcal{O}_{X}^{\times}))} \to {\it H_{\tau}^{2}(G; X; \mathcal{O}_{X}^{\times})}.
 \end{equation}\normalsize
 
 We say that a ring $A$ is seminormal if the following holds: whenever $b, c \in A$ satisfy $b^3=c^2$ there exists $a\in A$ such that $b=a^2,$ $c=a^3.$ For example, $\mathbb{C}[t^3-t^2, t^2-t]$ is a seminormal ring. A seminormal ring is necessarily reduced. Seminormality is a local property, i.e., $A$ is seminormal if and only if $A_{p}$ is seminormal for all $p\in Spec(A).$  A scheme $X$ is said to be seminormal if $\mathcal{O}_{X, x}$ is seminormal for all $x\in X.$ Equivalently, $X$ is seminormal if $\Gamma(U, \mathcal{O}_{X})$ is seminormal for each affine open subset $U$ of $X.$ For more details, we refer to \cite{Greco-trv}, \cite{swan}.
 
 \begin{proposition}\label{quotient s.n}
  Let $X$ be an integral quasi-projective $G$-scheme. If $X$ is seminormal then so is $X/G.$
  
  \begin{proof}
   Since $X$ is quasi-projective, the quotient map $q: X \to X/G$ is finite with  $\mathcal{O}_{X/G}= q_{*}(\mathcal{O}_{X})^{{\it G}}.$  We have to show that for every affine open subset $U$ of $X/G,$ $\Gamma(U, \mathcal{O}_{X/G})=\Gamma(X \times_{X/G} U, \mathcal{O}_{X})^{{\it G}}$ is seminormal. 
   Write $A= \Gamma(X \times_{X/G} U, \mathcal{O}_{X}).$ Note that $A$ is a seminormal domain. Let $b, c \in A^{{\it G}}\subset A$ such that $b^3=c^2.$ We may assume that $b,c\neq 0.$ Since $A$ is seminormal, there exists $0\neq a\in A$ such that $b= a^2,$ $c=a^3.$ Pick any $g\in G.$ Let $d= g.a \neq 0$  because $b, c \neq 0.$ Then $d^{3}a^{3}= d^{2}a^{4}.$ We get $g.a=d=a.$ This means that $a \in A^{{\it G}}.$  Hence, $A^{{\it G}}$ is seminormal.
  \end{proof}
\end{proposition}

Next, we discuss the homotopy invariance of the functor $\Pic^{\it G}.$ Recall the notations from section \ref{base section},  $N\Pic({\it X})= \ker [\Pic({\it X[t]}) \to \Pic({\it X})],$ where $X[t]$ denotes $X \times_{k} \mathbb{A}_{k}^{1}.$ Similarly, we have $N{\rm Pic}^{\it G}(X).$ Since $\pi: X[t] \to X$ admits an equivariant section, we have the formula (\ref{formula for pol}) for the functor ${\rm Pic}^{\it G}.$

The affine version of the following lemma is well known (see \cite{swan}). For lack of a reference, we include a proof for non affine version. 
\begin{lemma}\label{swan for scheme}
 Let $X$ be a seminormal scheme. Then the natural map $$\Pic({\it X}) \to \Pic({\it X[t_1, t_2, \dots, t_r]})$$ is an isomorphism for all $r\geq 0.$
\end{lemma}
\begin{proof}
 Since $X$ is seminormal, so is $X\times \mathbb{A}^{r}$ (see Corollary A.1 of \cite{Coq}).  So, it is enough to show that $N\Pic({\it X})=0.$ We know that a seminormal ring is necessarily reduced. Thus, $X$ is reduced. Let $\mathcal{NP}ic_{zar}$ be the Zariski sheafification of the presheaf $U\mapsto N\Pic({\it U}).$ Since $X$ is seminormal, $\mathcal{NP}ic_{zar}=0$ (see  \cite{swan} or Theorem 2.4 of \cite{Coq}). By Theorem 4.7 of \cite{Wei}, $N\Pic({\it X})={\it N}\Pic({\it X_{red}})= {\it H}_{zar}^{0}({\it X}, \mathcal{NP}ic)=0.$
\end{proof}

\begin{lemma}\label{Npic^G=NPic}
 Let $X$ be a reduced $G$-scheme. Then  $N{\rm Pic}^{\it G}({\it X})\cong ({\it N}\Pic({\it X}))^{{\it G}}.$
\end{lemma}
\begin{proof}
 Since X is reduced, $H^{0}(X, \mathcal{O}_{X}^{\times})=H^{0}(X[t], \mathcal{O}_{X[t]}^{\times}) .$ By comparing the sequence (\ref{5-term 2}) for $\mathcal{O}_{X}^{\times}$ and $\mathcal{O}_{X[t]}^{\times},$ we get $N{\rm Pic}^{\it G}({\it X})\cong ({\it N}\Pic({\it X}))^{{\it G}}.$
\end{proof}

\begin{theorem}
 \label{hom inv for equipic} Let $X$ be a seminormal $G$-scheme. Then the natural map $$ {\rm Pic}^{\it G}({\it X})\to {\rm Pic}^{\it G}({\it X[t_1, t_2, \dots, t_r]})$$ is an isomorphism for all $r\geq 0.$
\end{theorem}

\begin{proof}
 The assertion follows from Lemmas \ref{swan for scheme} and \ref{Npic^G=NPic}.
\end{proof}

\section{Equivariant Picard group is contracted}\label{main sec}
The main goal of this section is to show that the functor ${\rm Pic}^{\it G}$ is contracted in the sense of Bass. 
\begin{proposition}\label{equi MV seq}
 Let $X$ be a $G$-scheme. Suppose that $X= U\cup V$ for two open $G$-invariant subspaces $U$ and $V$ of $X.$ Given a $G$-module $\mathcal{F},$ there is a long exact $\tau$-cohomology sequence (here $\tau \in \{zar, et, nis\})$ \small
 $$ 0 \to H^{0}(G; X; \mathcal{F}) \to H^{0}(G; U; \mathcal{F})\oplus H^{0}(G; V; \mathcal{F}) \to H^{0}(G; U\cap V; \mathcal{F}) \to H^{1}(G; X; \mathcal{F}) \to \dots .$$\normalsize
\end{proposition}

\begin{proof}
  
 Let $ 0 \to \mathcal{F} \to \mathcal{I}^{0} \to \mathcal{I}^{1} \to \dots  $ be an injective resolution on the category of $G$-sheaves. For an invariant open subspace $W$ of $X,$ $j_{W}: W \hookrightarrow X$ denotes the natural $G$-linear inclusion. Write $\mathbb{Z}_{U}= j_{U!}j_{U}^{*}\mathbb{Z},$  $\mathbb{Z}_{V}= j_{V!}j_{V}^{*}\mathbb{Z}$ and  $\mathbb{Z}_{U\cap V}= j_{U\cap V !}j_{U\cap V}^{*}\mathbb{Z},$ where $\mathbb{Z}$ is the constant sheaf with trivial action. We have an exact sequence of $G$-sheaves
 \begin{equation}\label{exact in cons}
   0 \to \mathbb{Z}_{U\cap V} \to \mathbb{Z}_{U} \oplus \mathbb{Z}_{V} \to \mathbb{Z} \to 0.
 \end{equation}
Note that $Hom_{{\it G}}(\mathbb{Z}, \mathcal{I})\cong (\mathcal{I}(X))^{{\it G}}.$ Applying $Hom_{{\it G}}(-, \mathcal{I}^{\bullet})$ to the exact sequence (\ref{exact in cons}), we get a short exact sequence of complexes
$$0\to (\mathcal{I}^{\bullet}(X))^{{\it G}} \to (\mathcal{I}^{\bullet}(U))^{{\it G}}\oplus (\mathcal{I}^{\bullet}(V))^{{\it G}}\to (\mathcal{I}^{\bullet}(U\cap V))^{{\it G}} \to 0.$$ By taking the cohomology we get the assertion.\end{proof}

\subsection{Fact:}\label{import} Suppose $f: \mathcal{F} \to \mathcal{F^{'}}$ is an isomorphism $\tau$-sheaves on $X$ and $\phi: \sigma^{*}\mathcal{F} \cong Pr_{2}^{*}\mathcal{F}$ is a $G$-linearization of $\mathcal{F},$ where $\tau \in \{zar, et, nis\}.$ Then $\phi^{'}=\sigma^{*}(f)^{-1}\circ \phi \circ Pr_{2}^{*}(f)$ defines a $G$-linearization of $\mathcal{F}^{'}.$ In fact, $f: (\mathcal{F}, \phi) \to (\mathcal{F}^{'}, \phi^{'})$ is a $G$-equivariant isomorphism.

Let $F$ be a functor from the category of schemes (resp. $G$-schemes) to abelian groups. Recall that $LF(X)= \coker [F(X[t]) \oplus F(X[t^{-1}]) \to F(X[t, t^{-1}])]$(see section \ref{base section}). 
\begin{theorem}\label{Fundamental exact seq}
 Let $X$ be a $G$-scheme. Then  ${\rm Pic}^{\it G}(\mathbb{P}_{X}^{1})\cong {\rm Pic}^{\it G}({\it X})\oplus ({\it H}^{0}({\it X}, \mathbb{Z}))^{{\it G}}$ and there is a natural exact sequence
 \begin{equation}\label{explicit seq}
  0\to {\rm Pic}^{\it G}({\it X}) \to {\rm Pic}^{\it G}({\it X[t]})\oplus {\rm Pic}^{\it G}({\it X[t^{-1}]}) \to {\rm Pic}^{\it G}({\it X[t, t^{-1}]})\to {\it L}{\rm Pic}^{\it G}({\it X}) \to 0.
 \end{equation}
\end{theorem}
\begin{proof}
 Let $\pi: \mathbb{P}_{X}^{1} \to X$ be the structure $G$-linear morphism. By the above fact (\ref{import}) and the proof of Proposition 7.3 of \cite{Wei}, $\pi_{*}\mathcal{O}_{\mathbb{P}_{X}^{1}}^ {\times}\cong \mathcal{O}_{X}^{\times}$ as $G$-sheaves. Then $H^{0}(X, \mathcal{O}_{X}^{\times})\cong H^{0}(\mathbb{P}_{X}^{1}, \mathcal{O}_{\mathbb{P}_{X}^{1}}^ {\times})$ as $G$-modules. Consider the following commutative diagram (by (
 \ref{5-term 2}))
\tiny  $$\begin{CD}
    0 @>>> H^{1}(G,H^{0}(X, \mathcal{O}_{X}^{\times})) @>>> {\rm Pic}^{\it G}({\it X}) @>>> (\Pic({\it X}))^{{\it G}} @>>>  H^{2}(G,H^{0}(X, \mathcal{O}_{X}^{\times}))\\
    @.   @VV \cong V            @VVV            @VVV          @VV \cong V   \\
     0 @>>> H^{1}(G,H^{0}(\mathbb{P}_{X}^{1}, \mathcal{O}_{\mathbb{P}_{X}^{1}}^ {\times})) @>>> {\rm Pic}^{\it G}({\it \mathbb{P}_{X}^{1}}) @>>> (\Pic({\it \mathbb{P}_{X}^{1}}))^{{\it G}} @>>>  H^{2}(G,H^{0}(\mathbb{P}_{X}^{1}, \mathcal{O}_{\mathbb{P}_{X}^{1}}^ {\times})).
   \end{CD}$$\normalsize
We know $\Pic(\mathbb{P}_{{\it X}}^{1})\cong \Pic({\it X})\oplus {\it H}^{0}({\it X}, \mathbb{Z})$ (see Proposition 7.3 of \cite{Wei}). Thus, we get an exact sequence 
$$ 0 \to {\rm Pic}^{\it G}({\it X}) \to {\rm Pic}^{\it G}( \mathbb{P}_{X}^{1}) \to ({\it H}^{0}({\it X}, \mathbb{Z}))^{{\it G}} \to 0.$$  Set $U(X)= H^{0}(X, \mathcal{O}_{X}^{\times})$ for any scheme $X.$ Further, by applying Proposition \ref{equi MV seq} to the $G$-sheaf $\mathcal{O}_{\mathbb{P}_{X}^{1}}^{\times}$ with the $G$-invariant subspaces $X[t]$ and $X[t^{-1}]$ of $\mathbb{P}_{X}^{1},$ we obtain 
 \begin{multline}\label{upic long}
   0 \to (U(\mathbb{P}_{X}^{1}))^{{\it G}} \to (U(X[t]))^{{\it G}} \oplus (U(X[t^{-1}]))^{{\it G}} \to (U(X[t, t^{-1}]))^{{\it G}} \to {\rm Pic}^{\it G}({\it \mathbb{P}_{X}^{1}})\to \\ {\rm Pic}^{\it G}({\it X[t]})\oplus {\rm Pic}^{\it G}({\it X[t^{-1}]}) \to {\rm Pic}^{\it G}({\it X[t, t^{-1}]})\to \dots.
 \end{multline}
We have $U( \mathbb{P}_{X}^{1})\cong U(X)$ and $U(X[t, t^{-1}])\cong U(X) \oplus NU(X) \oplus NU(X) \oplus {\it H}^{0}({\it X}, \mathbb{Z})$ (see Propositions 7.2 and 7.3 of \cite{Wei}). Therefore, the above sequence (\ref{upic long}) reduces to 
\small
$$ 0 \to ({\it H}^{0}({\it X}, \mathbb{Z}))^{{\it G}} \stackrel{\partial}\to {\rm Pic}^{\it G}( \mathbb{P}_{X}^{1})\to \\ {\rm Pic}^{\it G}({\it X[t]})\oplus {\rm Pic}^{\it G}({\it X[t^{-1}]}) \to {\rm Pic}^{\it G}({\it X[t, t^{-1}]})\to \dots.$$ \normalsize Note that the map $\partial $ is the right inverse of the map ${\rm Pic}^{\it G}( \mathbb{P}_{X}^{1}) \to ({\it H}^{0}({\it X}, \mathbb{Z}))^{{\it G}}.$ Hence the assertion.
\end{proof}

 Let $f: X \to S$ be a morphism between $G$-schemes, i.e., $G$-linear morphism. For an \'etale $G$-sheaf $\mathcal{F}$ on $X,$ $f_{*}\mathcal{F}$ is a $G$-sheaf on $S.$ In fact, $f_{*}$ sends injective $G$-sheaves on $X$ to injective $G$-sheaf on $S$ (see pp. 499 of \cite{FU}). Now the following  commutative diagram 
 \label{commutativity1}
     $$\xymatrixcolsep{5pc}\xymatrix{
 Ab_{\tau}(G, X) \ar[rd]^{\Gamma_{S}^{{\it G}}} \ar[r]^{f_{*}} & Ab_{\tau}(G, X) \ar[d]^{\Gamma_{X}^{{\it G}}} \\
 &  Ab} $$of categories gives a first quadrant convergent spectral sequence
\begin{equation}\label{spec seq 2}
 E_{2}^{pq}= H_\mathrm{et}^{p}(G, S, R^{q}f_{*}\mathcal{F})\Rightarrow H_\mathrm{et}^{p+q}(G; X; \mathcal{F}).
\end{equation}

\begin{remark}\label{true for zar and nis}{\rm
 The spectral sequence (\ref{spec seq 2}) also exists for the Zariski or Nisnevich topology. We stated the \'etale version as it will be used in Theorem \ref{Picg contracted} below.} 
\end{remark}

Write $Y= X[t, t^{-1}],$ $Y^{+}= X[t],$ and $Y^{-}= X[t^{-1}].$ Let $\pi$ (resp. $\pi^{+},$ $\pi^{-}$) denote the structure $G$-linear map $Y \to X$ (resp. $Y^{+} \to X,$ $Y^{-} \to X$). We have the following isomorphisms of \'etale $G$-sheaves on $X$(see Proposition 7.2 of \cite{Wei} and Fact (\ref{import})) 
\begin{equation}\label{nat split unit sheaves}
 \pi_{*}^{+}\mathcal{O}_{Y^{+}}^{\times}\cong \mathcal{O}_{X}^{\times} \times N\mathcal{O}_{X}^{\times},~~ \pi_{*}\mathcal{O}_{Y}^{\times}\cong \mathcal{O}_{X}^{\times} \times N\mathcal{O}_{X}^{\times} \times N\mathcal{O}_{X}^{\times} \times \mathbb{Z}.
\end{equation}

\begin{theorem}\label{Picg contracted}
 ${\rm Pic}^{\it G}$ is a contracted functor on the category $Sch^{{\it G}}/k$ with $L{\rm Pic}^{\it G}({\it X})\cong {\it H}_\mathrm{et}^{1}(G; X; \mathbb{Z}).$ Moreover, the splitting map is given by:
 $$L{\rm Pic}^{\it G}({\it X})\cong {\it H}_\mathrm{et}^{1}(G; X; \mathbb{Z}) \to {\it H}_\mathrm{et}^{1}(G; X; \pi_{*}\mathcal{O}_{X[t, t^{-1}]}^{\times}) \to {\rm Pic}^{\it G}({\it X[t, t^{-1}]}).$$
\end{theorem}

\begin{proof}
 We need to show that the exact sequence (\ref{explicit seq})(see Theorem \ref{Fundamental exact seq}) is naturally split for any $G$-scheme $X.$ Let $\mathcal{P}ic[T]$ (resp. $\mathcal{NP}ic$) denote the \'etale sheaf on $X$ associated to the presheaf $U \mapsto \Pic({\it U[t, t^{-1}]})$
 (resp. $ U \mapsto N\Pic({\it U})$). The stalk of  $\mathcal{P}ic[T]$ (resp. $\mathcal{NP}ic$) at a geometric point $\bar{x}$ of $X$ is $\Pic({\it \mathcal{O}_{X, x}^{sh}[t, t^{-1}}])$ (resp. $\it{N}\Pic({\it \mathcal{O}_{X, x}^{sh}})$), where $\mathcal{O}_{X, x}^{sh}$ is the strict henselization of the corresponding local ring $\mathcal{O}_{X,x}.$ There is an exact sequence (see Proposition 7.3 of \cite{Wei})
 $$ 0 \to \Pic({\it X})\oplus {\it N}\Pic({\it X})\oplus {\it N}\Pic({\it X}) \to \Pic({\it X[t, t^{-1}}]) \to {\it L}\Pic({\it X})\to 0$$ for any scheme $X.$ Moreover,  $L\Pic({\it R})=0$ for a hensel local ring $R$ (see Theorem 2.5 of \cite{Wei}). This gives that   $\mathcal{P}ic[T]\cong \mathcal{NP}ic\oplus \mathcal{NP}ic$ as an \'etale sheaves (see Proposition 5.1 of \cite{Wei}).  The stalk $(R^{1}\pi_{*}^{+}\mathcal{O}_{Y^{+}}^{\times})_{\bar{x}}$ ~(resp. $(R^{1}\pi_{*}^{+}\mathcal{O}_{Y}^{\times})_{\bar{x}}$) at a geometric point $\bar{x}$ of $X$ is isomorphic to $\Pic({\it \mathcal{O}_{X, x}^{sh}[t]})$ (resp. $\Pic({\it \mathcal{O}_{X, x}^{sh}[t, t^{-1}]}).$ Note that both $R^{1}\pi_{*}^{+}\mathcal{O}_{Y^{+}}^{\times}$ and $R^{1}\pi_{*}^{+}\mathcal{O}_{Y}^{\times}$ are $G$-sheaves. Hence by using the Fact (\ref{import}), we get
 $$R^{1}\pi_{*}^{+}\mathcal{O}_{Y^{+}}^{\times}\cong \mathcal{NP}ic ~
 {\rm and}~ R^{1}\pi_{*}\mathcal{O}_{Y}^{\times}\cong \mathcal{P}ic[T]\cong \mathcal{NP}ic\oplus \mathcal{NP}ic$$ as $G$-sheaves. We compare the spectral sequences (\ref{spec seq 2}) for $\pi,$ $\pi^{+}$ and $\pi^{-},$  and get the following exact diagram (using Theorem \ref{Fundamental exact seq})
  \tiny
 $$\begin{CD} @. 0  @. 0 \\
 @.  @VVV   @VVV  \\
  0 @>>> {\rm Pic}^{\it G}({\it X})\oplus{\it  H_\mathrm{et}^{1}(G; X, \mathcal{N}\mathcal{O}_{X}^{\times})}\oplus {\it H_\mathrm{et}^{1}(G; X, \mathcal{N}\mathcal{O}_{X}^{\times})} @>>> H_\mathrm{et}^{1}(G;X, \pi_{*}\mathcal{O}_{Y}^{\times}) @>>> H_\mathrm{et}^{1}(G;X, \mathbb{Z}) @>>> 0\\
  @. @VVV      @VVV     @VVV  \\
  0 @>>> {\rm Pic}^{\it G}({\it X})\oplus {\it N}{\rm Pic}^{\it G}({\it X}) \oplus {\it N}{\rm Pic}^{\it G}({\it X}) @>>> {\rm Pic}^{\it G}({\it X[t, t^{-1}]}) @>>> L{\rm Pic}^{\it G}({\it X}) @>>> 0 \\
  @. @VVV  @VVV @. \\
  @. H_\mathrm{et}^{0}(G; X; \mathcal{NP}ic)\oplus H_\mathrm{et}^{0}(G; X; \mathcal{NP}ic) @>>=> H_\mathrm{et}^{0}(G; X; \mathcal{NP}ic)\oplus H_\mathrm{et}^{0}(G; X; \mathcal{NP}ic)\\
  @. @VVV @VVV @. \\
  0 @>>> H_\mathrm{et}^{2}(G; X; \mathcal{O}_{X}^{\times} \times N\mathcal{O}_{X}^{\times} \times N\mathcal{O}_{X}^{\times}) @>>> H_\mathrm{et}^{2}(G; X; \mathcal{O}_{X}^{\times} \times N\mathcal{O}_{X}^{\times} \times N\mathcal{O}_{X}^{\times}\times \mathbb{Z}) 
 \end{CD}.$$
\normalsize
 Observe that the first row of the above commutative diagram is naturally split exact (see (\ref{nat split unit sheaves})). A diagram chase implies that ${\it H}_\mathrm{et}^{1}(G; X; \mathbb{Z})\cong L{\rm Pic}^{\it G}({\it X}).$  Finally, we obtain our desired splitting $L{\rm Pic}^{\it G}({\it X}) \to {\rm Pic}^{\it G}({\it X[t, t^{-1}]})$ as the composition of the following maps
 $L{\rm Pic}^{\it G}({\it X})\stackrel{\cong}\to {\it H}_\mathrm{et}^{1}(G; X; \mathbb{Z}) \to {\it H}_\mathrm{et}^{1}(G; X; \pi_{*}\mathcal{O}_{Y}^{\times}) \to {\rm Pic}^{\it G}({\it X[t, t^{-1}]}).$\end{proof}

\begin{remark}\label{true for nis but not zar}{\rm 
 The argument given above for the \'etale topology also works for the Nisnevich topology (see Remark \ref{true for zar and nis}). Therefore, we get $L{\rm Pic}^{\it G}({\it X})\cong {\it H}_\mathrm{nis}^{1}(G; X; \mathbb{Z}).$ The result may fail for the Zariski topology. For example, consider the nodal curve $X= {\rm Spec}(k[x, y]/ (y^2-x^2-x^3))$ with $\mathbb{Z}_{2}$-action given by $(x, y)\mapsto (x, -y).$ Then $H_\mathrm{et}^{1}(X, \mathbb{Z})\cong L\Pic(X)=\mathbb{Z}$ (see Remark 5.5.2 of \cite{Wei}) and $H_{zar}^{1}(X, \mathbb{Z})=0$ because $X$ is an integral scheme. By considering the exact sequence (\ref{5-term 1}) for the Zariski topology with constant sheaf $\mathbb{Z},$ we get
 
 $$0 \to H^{1}(\mathbb{Z}_{2}, \mathbb{Z}) \to H_{zar}^{1}(\mathbb{Z}_{2}; X; \mathbb{Z}) \to H^{0}(\mathbb{Z}_{2}, H_{zar}^{1}(X, \mathbb{Z})) \to \dots .$$ Note that $H^{1}(\mathbb{Z}_{2}, \mathbb{Z})=0.$ Hence, $H_{zar}^{1}(\mathbb{Z}_{2}; X; \mathbb{Z})=0.$ On the other hand, the exact sequence (\ref{5-term 1}) for the \'etale topology implies the following 
 $$0 \to H^{1}(\mathbb{Z}_{2}, \mathbb{Z}) \to H_\mathrm{et}^{1}(\mathbb{Z}_{2}; X; \mathbb{Z}) \to H^{0}(\mathbb{Z}_{2},  \mathbb{Z}) \to H^{2}(\mathbb{Z}_{2},  \mathbb{Z}) .$$ Therefore, $L\Pic^{{\it \mathbb{Z}_{2}}}({\it X})\cong {\it H}_\mathrm{et}^{1}(\mathbb{Z}_{2}; {\it X}; \mathbb{Z})\neq 0$ because $H^{2}(\mathbb{Z}_{2}, \mathbb{Z})= \mathbb{Z}_{2}$ (see Example 6.2.3 of \cite{Weihom}).}
\end{remark}

\begin{theorem}\label{NL^G=0}
 For a $G$-scheme $X,$  $L{\rm Pic}^{\it G}({\it X})\cong {\it L}{\rm Pic}^{\it G}({\it X[t]})\cong {\it L}{\rm Pic}^{\it G}({\it X[t, t^{-1}]}).$ In otherwords, $NL{\rm Pic}^{\it G}({\it X})= {\it L}^{2}{\rm Pic}^{\it G}({\it X})=0.$
\end{theorem}
\begin{proof}
 We have (see Example \ref{ex1}) $$H_\mathrm{et}^{0}(X, \mathbb{Z})\cong H_\mathrm{et}^{0}(X[t], \mathbb{Z})\cong H_\mathrm{et}^{0}(X[t, t^{-1}], \mathbb{Z})$$ and $$L\Pic({\it X})\cong {\it L}\Pic({\it X[t]})\cong {\it L}\Pic({\it X[t, t^{-1}}]).$$ By comparing the sequence (\ref{5-term 1}) for $X,$ $X[t]$ and $X[t, t^{-1}]$ with constant $G$-sheaf $\mathbb{Z},$ we get ${\it H}_\mathrm{et}^{1}(G; X; \mathbb{Z})\cong {\it H}_\mathrm{et}^{1}(G; X[t]; \mathbb{Z})\cong {\it H}_\mathrm{et}^{1}(G; X[t, t^{-1}]; \mathbb{Z}).$ Hence the assertion by Theorem \ref{Picg contracted}.
\end{proof}

\begin{corollary}\label{formula for Pic^g}
 Let $X$ be a $G$-scheme. Then there is a natural decomposition\small
 $${\rm Pic}^{\it G}({\it X[t_{1}, t_{1}^{-1}, t_{2}, t_{2}^{-1}, \dots, t_{m}, t_{m}^{-1}}])\cong {\rm Pic}^{\it G}({\it X}) \oplus \coprod_{k=1}^{m} \coprod_{i=1}^{2k {m \choose k}}{\it N}^{k}{\rm Pic}^{\it G}({\it X})\oplus \coprod_{i=1}^{m}{\it H}_\mathrm{et}^{1}({\it G; X}; \mathbb{Z}).$$ \normalsize
\end{corollary}

\begin{proof}
 By Theorem \ref{Picg contracted}, ${\rm Pic}^{\it G}$ is a contracted functor on $Sch^{{\it G}}/k.$ The result is now clear from (\ref{gen form}) and Theorem \ref{NL^G=0}.
\end{proof}

Given a $G$-scheme, there is a natural homomorphism $\eta_{X}: {\rm Pic}^{\it G}(X) \to \Pic(X)$ sending   $[(\mathcal{L}, \phi)]$ to $[\mathcal{L}].$ So, we get a natural transformation $\eta: {\rm Pic}^{\it G} \Rightarrow \Pic$ of functors on  $Sch^{{\it G}}/k.$ By Fact \ref{import}, the kernel of $\eta_{X}$ consists of the isomorphism classes of $G$-linearization of $\mathcal{O}_{X}.$ The exact sequence (\ref{5-term 2}) implies that $\ker(\eta_{X})\cong H^{1}(G, H_{\tau}^{0}(X, \mathcal{O}_{X}^{\times})),$ where $\tau \in \{zar, et, nis\}.$  Next, we show that the kernel $\ker(\eta)$ is a contracted functor on $Sch^{{\it G}}/k.$

Let $F$ and $F^{'}$ be two contracted functors on some category $\mathcal{C}$ (e.g., rings or schemes). A morphism between contracted functors $F$ and $F^{'}$ is a natural transformation $\eta: F \Rightarrow F^{'}$ such that the following diagram commutes 
$$\begin{CD}
 LF(X) @>>> F(X[t, t^{-1}])\\
 @V L\eta_{X}VV    @V \eta_{X[t, t^{-1}]}VV \\
 LF^{'}(X) @>>> F^{'}(X[t, t^{-1}])
\end{CD}$$ for all $X\in {\rm {ob}}(\mathcal{C}).$

\begin{lemma}\label{forgetful map is nat trns}
 The natural transformation $\eta: {\rm Pic}^{\it G} \Rightarrow \Pic$ is a morphism of contracted functors on $Sch^{{\it G}}/k.$
\end{lemma}
\begin{proof} Write $Y= X[t, t^{-1}]$ and $\pi: Y \to X.$ The result follows from the following commutative diagram \small
$$\begin{CD}
  L{\rm Pic}^{\it G}({\it X})\cong {\it H}_\mathrm{et}^{1}(G; X; \mathbb{Z}) @>>> H_\mathrm{et}^{1}(G; X, \pi_{*}\mathcal{O}_{Y}^{\times}) @>>> H_\mathrm{et}^{1}(G; Y, \mathcal{O}_{Y}^{\times})\cong {\rm Pic}^{\it G}(Y)\\
  @VVV    @VVV            @VVV \\
  H^{0}(G, H_\mathrm{et}^{1}(X, \mathbb{Z})) @. H^{0}(G, H_\mathrm{et}^{1}( X, \pi_{*}\mathcal{O}_{Y}^{\times})) @. H^{0}(G, H_\mathrm{et}^{1}(Y, \mathcal{O}_{Y}^{\times}))\\
  @V{\rm inclusion}VV     @V{\rm inclusion}VV   @V{\rm inclusion}VV  \\
  L\Pic({\it X})\cong {\it H_\mathrm{et}^{1}(X, \mathbb{Z})} @>>> H_\mathrm{et}^{1}(X, \pi_{*}\mathcal{O}_{Y}^{\times}) @>>> H_\mathrm{et}^{1}(Y, \mathcal{O}_{Y}^{\times})\cong \Pic({\it Y}),
 \end{CD}$$\normalsize
where the top and bottom rows are precisely the splitting maps for ${\rm Pic}^{\it G}$ and $\Pic$ respectively (see Theorem \ref{Picg contracted} and Theorem 7.6 of \cite{Wei}), and for columns see (\ref{5-term 1}).
\end{proof}

\begin{theorem}\label{ker is contracted}
 The functor $\ker(\eta)$ is contracted on $Sch^{{\it G}}/k.$ Moreover, the contraction term $L\ker(\eta_{X})$ is isomorphic to $H^{1}(G, H_\mathrm{et}^{0}(X, \mathbb{Z})).$
\end{theorem}
\begin{proof}
 By Lemma \ref{forgetful map is nat trns}, $\eta$ is a morphism of contracted functors on $Sch^{{\it G}}/k.$ Then $\ker(\eta)$ is a contracted functor with contraction $L\ker(\eta)(X)=L\ker(\eta_{X})$ for any $X\in Sch^{{\it G}}/k$ (see Lemma 2.2 of \cite{Sw}). The second statement follows from (\ref{5-term 1}).
\end{proof}

We conclude by discussing a vanishing result of $L{\rm Pic}^{\it G}({\it X})$ for a hensel local ring.

\begin{theorem}\label{vanishing of LPic^G}
 Let $X$ be a connected $G$-scheme.  If $L\Pic({\it X})=0$ then $L{\rm Pic}^{\it G}({\it X})=0.$
  \end{theorem}
  
 \begin{proof}
The  exact sequence (\ref{5-term 1}) for the constant $G$-sheaf $\mathbb{Z}$ (using Example \ref{ex1} and Theorem \ref{Picg contracted}) gives that
$$ 0 \to H^{1}(G, H_\mathrm{et}^{0}(X, \mathbb{Z})) \to L{\rm Pic}^{\it G}({\it X}) \to ({\it L}\Pic({\it X}))^{{\it G}} \to \dots .$$ Since $X$ is connected,  $H_\mathrm{et}^{0}(X, \mathbb{Z})\cong H_{zar}^{0}(X, \mathbb{Z}) \cong \mathbb{Z}.$ Thus, $H^{1}(G, \mathbb{Z})=0$ because $G$ is a finite group. 
Hence the assertion. \end{proof}

\begin{corollary}\label{vanishing for hensel}
 Let $X={\rm Spec}(A)$ be a $G$-scheme. If $A$ is a hensel local ring then $L{\rm Pic}^{\it G}({\it X})=0.$ 
\end{corollary}
\begin{proof}
 For a hensel local ring $A,$ $L\Pic({\it A})=0$ (see Theorem 2.5 of \cite{Wei}). Hence the result by Theorem \ref{vanishing of LPic^G}. 
\end{proof}

\end{document}